\documentclass[11pt,a4paper]{amsart}
\usepackage[utf8]{inputenc}

\usepackage{doc}
\def\thanks#1{{\let\thefootnote\relax\footnote{#1.}\setcounter{footnote}{0}}}
\makeatletter
\newcommand\thankssymb[1]{\textsuperscript{\@fnsymbol{#1}}}

\usepackage{amssymb}
\usepackage{amscd}
\usepackage{amsmath}
\usepackage{eucal}
\usepackage{graphicx}
\usepackage{verbatim}
\usepackage{amsthm}
\usepackage{hyperref}
\usepackage{mathrsfs}
\usepackage{color}
\usepackage{multicol}
\usepackage{tikz}
\usepackage{float}
\usepackage{wrapfig}
\usepackage{geometry}
\usepackage{tikz-cd} 
\usetikzlibrary{arrows,intersections}
\usetikzlibrary{patterns}
\usepackage{epigraph}
\usepackage{capt-of}

\hypersetup{
    colorlinks=true,
    citecolor=blue
           }

\newcommand{\Mod}[1]{\ (\mathrm{mod}\ #1)}

\newtheorem{theorem}{Theorem}
\newtheorem{lemma}{Lemma}

\theoremstyle{definition}
\newtheorem{definition}{Definition}

\theoremstyle{remark}

\title{Modular Arrangements}
\author{A. Levin %\thankssymb{1}
} \author{N. Sakharova %\thankssymb{2}
}
\date{}

\begin{document}
%\centerline {\bf Modular Arrangements }
\maketitle
%\thanks{\thankssymb{1} The  author was supported by the Basic Research Program of the National Research University Higher School of Economics} 
%\thanks{\thankssymb{2} The  author was supported by the Basic Research Program of the National Research University Higher School of Economics}
\section{Introduction}

The modular curves are fine objects for testing conjectures in the arithmetic geometry. They have natural geometric definition in contrast with rather nontrivial structure. At the other hand, they are very good studied from the number-theoretical viewpoint. In addition,  there is well developed powerful analytical technique.
We shall use the square of the modular curve as the experimental object for study arithmetical properties of the periods of the mixed Hodge structures. There is extra reason for this problem: this square is naturally equipped by the family of (Hecke) curves. They are components of the Neron-Severi locus for interpretation of the square of modular curves as moduli space of splitted abelian surfaces (namely, a product of two elliptic curves).

In the article we construct explicitly meromorphic differential forms with logarithmic singularities along the Hecke curves. For this we use two ideas: 

First, the maps between curves of different levels. For suitable combination of this maps the Hecke curve is presented as the image of the diagonal.

Second, the realization of the Modular Cauchy kernel as Zagier series.

For the simplest case we reduce some period of geometrical mixed Hodge structure. This reduction is based on the reciprocity low map according to Rudenko \cite{Ru15}.
  
\section{1-2-0 of the  Elliptic Modularity}
\subsection{The Generalities of the Modular Curves.} 

Fix some notations. The upper half-plane ${\Bbb H}=\{z\in \Bbb C~|~\Im(z)>0\}$,
$$\Gamma={\rm SL}_2(\Bbb Z)=\left\{\left.\left(\begin{matrix}a&b\\c&d\end{matrix}\right)\right|ad-bc=1\right\}.$$
For any matrix $M=\left(\begin{matrix}a&b\\c&d\end{matrix}\right)$ denote $M^*= \left(\begin{matrix}d&-b\\-c&a\end{matrix}\right). $
The modular curve $Y$ should be the space of orbits of linear fractional  action of the group $\Gamma$ on the upper half-plane (f.e. as orbifold). So we can work with objects on $\Bbb H$, which are compatible with the action of $\Gamma$.

The automorphic factor $cz+d$ is a 1-cocycle of $\Gamma$ with values in invertible functions on $\Bbb H$, so its powers determine twisted actions of the group $\Gamma$ on the space of function on $\Bbb H$, this corresponds to automorphic forms. Since $\partial\left((az+b)/(cz+d)\right)/\partial z= 1/(cz+d)^2$, so  $\Gamma$-invariant $(1,0)$-differential forms correspond to automorphic forms of the weight $2$. Bellow we call such objects just differential forms on the modular curve.\\

{\em Motivation.} The modular curve appears as the set of discrete 2-generate abelian  subgroups in complex line: $\Bbb Z^2=L\subset \Bbb C$, we shall call this object the {\em lattice}. The choice of the clockwise  generators of $L$ produces pair $(\omega_1,\omega_2)$ in the line $\Bbb C$; pic $\omega_2$ as the frame of $\Bbb C$. Hence  the fraction $z=\omega_1/\omega_2$ is the coordinate of the first generator of $L$. The  change of generators of $L$
$$
\left(\begin{matrix}   
  \omega_1    \\
  \omega_2    
\end{matrix}\right)\to
\left(\begin{matrix}  
 a  &b  \\
 c  & d
\end{matrix}\right)
\left(\begin{matrix}   
  \omega_1    \\
  \omega_2    
\end{matrix}\right)
$$
corresponds to the fractional-linear action of $\Gamma$ on $z$.\\

{\em  Modular Figure.} Put $\Phi^0=\{z\in \Bbb H~|~|z|>1,~ |\Re(z)|<\frac{1}{2}\}$,
$\overline{\Phi}=\{z\in \Bbb H~|~|z|\geq 1, ~|\Re(z)|\leq\frac{1}{2}\}$. Any $\Gamma$-orbit intersect $\overline{\Phi}$; indeed,  pick clockwise  generators of $L$  $(\omega_1,\omega_2)$ by the rule: $\omega_2$ is the shortest nonzero vector, $\omega_1$ is the shortest nonproportional  to $\omega_2$ vector. At the other hand,
 the intersection of any $\Gamma$-orbit with $\Phi^0$ contain no more that one point. Hence, one can use integration over 
$\overline{\Phi}$ as imitation of integration over $Y$.

There is a canonical compactification  of $Y$ by the $\Bbb H \cup \Bbb Q \Bbb P^1$ modulo $\Gamma$ based on the following remark. Denote the stabilizer
$$ \left\{\begin{pmatrix}
    \pm 1&\nu\\
    0&\pm 1 
\end{pmatrix}\right\}\subset \Gamma$$ of the point $\infty$ by $\Gamma^\infty$.
The space of $\Gamma^\infty$-orbits  $\Bbb H$ equals to punctured disc $0<q<1$ by the map $z\to q=\exp(2\pi i z)$. The (partial) compactification of the punctured disc by its center $q=0$ corresponds to the $\Gamma^\infty$-orbit of $\infty\in \Bbb Q \Bbb P^1$. This compact version of the modular curve is denoted by $X$.

\subsection{The Level Structure  and the Modular  Correspondences.}

Consider the following  subgroup of $\Gamma$:
$$\Gamma_0(n)=\left\{ g=\left(\begin{matrix}a&b\\c&d\end{matrix}\right)|c\equiv 0 \Mod{n}\right\}
\mbox{. Define}
\left(\begin{matrix}a&b\\nc'&d\end{matrix}\right)^{\diamond}=\left(\begin{matrix}d&- c'\\-nb&a\end{matrix}\right).$$
Put $W_n=\left(\begin{matrix}0&- 1\\n&0\end{matrix}\right)$. Then $g^{\diamond}=W_ngW_n^{-1}$, $\Gamma_0(n)=\Gamma\cap W_n\Gamma W_n^{-1}$.

The level structure modular curve $Y_0(n)$ is $\Bbb H$, equipped by the action of $\Gamma_0(n)$.
The canonical embedding $\Gamma_0(n)\to \Gamma$ corresponds to the map $\rho: Y_0(n)\to Y$. The map
$\Bbb H\to \Bbb H: z\to -\frac1{nz} =W_n(z)$ equipped by $\Gamma_0(n)\to \Gamma$, such that $g\to g^{\diamond}=W_ngW_n^{-1}$, corresponds to the map $\lambda: Y_0(n)\to Y$. 

The curve $Y_0(n)$ responds to {\em flags} of lattices $L_l\subset L_r\subset \Bbb C$ with cyclic quotient group  $L_r/L_l$ of order $n$ by corresponding $(\omega_1,\omega_2)$ the flag $\Bbb Z \omega_1\oplus \Bbb Z n\omega_2 \subset \Bbb Z \omega_1\oplus \Bbb Z \omega_2 $. The map $\lambda$ corresponds to the flag $L_l\subset L_r\subset \Bbb C$  its left component $L_l\subset  \Bbb C$ (with generators $(-\omega_2,n\omega_1)$); $\rho$ maps flag to $ L_r\subset \Bbb C$. 
 
The Fricke element $W_n$ centralize $\Gamma_0(n)$ and the map $\Bbb H\to \Bbb H: z\to W_n(z)$ equipped by $\Gamma_0(n)\to \Gamma_0(n):~ g\to W_ngW_n^{-1}$ determines a (Fricke) involution $W_n: Y_0(n)\to Y_0(n)$. 
The Fricke involution acts on flags as $(L_l\subset L_r)\to (n L_r\subset L_l)$.

It is possible to determine the compactification $X_0(n)$, let $X=X_0(1)$. 
 
The image of $\lambda\times \rho$ in $Y\times Y$ is known as modular or Hecke correspondence $T_n$.
The naive image $(-\frac1{nz},z)$ is not a curve in $Y\times Y$ as it is not $\Gamma\times \Gamma$-invariant, so we shall
pass to its $\Gamma\times \Gamma$-orbit, which has the following description. An integer matrix is called {\em primitive} if general common divisor of its entries equals $1$. Denote by ${\rm Mat}'_n$ the set of integer  primitive matrices of determinant $n$; for any matrix $M$ denote by   $\widetilde{T}_M\subset \Bbb H\times \Bbb H$ the curve $z_2=M(z _1)$. Then the Hecke correspondence is equal to the union of this curves over all primitive matrices  of the determinant $n$: $\bigcup_{M\in {\rm Mat}'_n} \widetilde{T}_M$. This is corollary of the following algebraic speculation. The group ${\rm SL}_2(\Bbb Z)$ acts by left or right multiplication on ${\rm Mat}'_n$: $L_g(M)=gM, R_g(M)=Mg^{-1}$. The resulting  left-right action of $\Gamma\times\Gamma$ on ${\rm Mat}'_n$ is transitive as the last set is an orbit of $W_n$, this is just the Smith normal form; the stabilizer of $M$ equals to the image of $\Gamma_0(n)$ under the map $g\to (g^{\diamond},g)$.

On the other hand, pushing down of $\widetilde{T}_M$ by  $\Gamma\times \Gamma$  is equivalent to action on $M$.\\

{\bf Remark.} The Hecke correspondence is symmetric. Indeed, for the lattice description this is  just  the following fact:
for 2-generated free abelian groups $L_1 \subset L_2$ if the quotient $L_2/L_1$ is cyclic of order $n$, then $nL_2\subset L_1$ and
the quotient  is cyclic of order $n$ too; for matrix description  this follows from the evident equation
     
$$gg^*=\begin{pmatrix}
a & b \\
c&  d
\end{pmatrix}\begin{pmatrix}
d & -b\\
-c & a
\end{pmatrix}=\begin{pmatrix}
ad-bc & 0\\
0& ad-bc
\end{pmatrix}=\rm{diag}(\rm{det}(g)),
     $$     
and  the second matrix is primitive iff the first is primitive, the determinants of these two matrices coincide.
     
\subsection{The CM-Points.}

Remind that if $\Bbb Z \ni \Delta<0$   and $\Delta\equiv 0,1 \Mod{4}$, then  ${\mathcal O}_\Delta =\omega_\Delta\Bbb Z\oplus \Bbb Z$,  $\omega_\Delta =\frac{\Delta+\sqrt{\Delta}}{2}$, is closed with respect to the multiplication and is known as the imaginary-quadratic order of the discriminant $\Delta$. Any discrete subring with unit of $\Bbb C$ is either the ring of integers $\Bbb Z$, or an imaginary-quadratic order.

Factorize  $\Delta=f^2\Delta_0$, where $\Delta_0$  is squarefree, then $\Delta_0$ is known as the fundamental discriminant and $f$ is known as the conductor. The fraction field $F_{\Delta_0}={\mathcal O}\otimes \Bbb Q$ depends on $\Delta_0$ only.

Any ideal ${\frak a}$ of ${\mathcal O}_\Delta$ is a lattice  ${\frak a}\subset\Bbb C$, so the class of ideals defines a  CM-point $P({\frak a})$  on $Y$ of the shape $\frac{-B+\sqrt{\Delta}}{2A}, B^2-4AC=\Delta$.

The modular correspondences intersect in $Y\times Y$. The point of intersection can be described as follows. 
Let ${\mathcal O}$ is an imaginary  quadratic order.
Let $\frak a$ and $\frak b$ are ideals of ${\mathcal O}$.  Consider the set ${\rm Isog}(F;{\frak a},{\frak b})$ of primitive elements of the lattice 
$ {\frak a}^{-1}{\frak b}=\{\mu\in F~|~\mu{\frak a}\subset{\frak b}\}$ 
modulo multiplication  by ${\mathcal O}r^*$. Evidently this set ${\rm Isog}(F;{\frak a},{\frak b})$ depends  on classes of ideals $[{\frak a}], [{\frak b}]$  only: 
$\mu \in{\rm Isog}(F;{\frak a},{\frak b})$ 
maps to 
$\beta\mu \alpha^{-1}\in{\rm Isog}(F;\alpha{\frak a},\beta{\frak b})$.
 Define the  index function ${\rm I}$ on set ${\rm Isog}(F;{\frak a},{\frak b})$ to $\Bbb N$  by the rule  
$${\rm I}(\mu)=[{\frak b}:\mu{\frak a}]={\rm Norm}(\mu)\frac{{\rm Norm}({\frak a})}{{\rm Norm}({\frak b})}.$$

For $\mu\in {\frak a}^{-1}{\frak b}$, $\mu{\frak a}\subset{\frak b}$ determines flag of lattices
$\mu{\frak a}\subset{\frak b}\subset \Bbb C$ of index ${\rm I}(\mu)$. Denote corresponding point on 
$Y_0({\rm I}(\mu))$ by $Q({\frak t}, \mu)$.

Any triple ${\frak t}= (F;[{\frak a}], [{\frak b}])$ corresponds to a point $P({\frak t})=(P({\frak a}), P({\frak b}))$ in $Y\times Y$ which belongs to $T_{{\rm I}(\mu)}$, where $\mu  \in{\rm Isog}({\mathcal O};[{\frak a}], [{\frak b}])$ determines a branch  of the curve 
$T_{{\rm I}(\mu)}$ passing through this point  $P({\frak t})$ under the isomorphism $Y_0({\rm I}(\mu))\to T_{{\rm I}(\mu)}$. The point $Q({\frak t},\mu )\in Y_0({\rm I}(\mu))\in T_{{\rm I}(\mu)}$ is the preimage of this intersection point  $P({\frak t})$. 

\newpage

 %%%%%%%%%%%%%%%%%%%%%%%%%%%%%%%%%%%%%%%%%%%%%%%%%%%%%

\section{The Arrangements}
\subsection{General Setup.}

 %%%%%%%%%%%%%%%%%%%%%%%%%%%%%%%%%%%%%%%%%%%%%%%%%%%%%

\begin{definition}
Modular or Hecke arrangements is a finite collection of the Hecke divisors in the square $X^2$ of the modular curve $X$.
\end{definition}

Rather standard  objects of research are holomorphic differential form on the complement to the arrangement with first order poles at the components of the arrangement. 

Let $\{C_a\}$ be a generic arrangement of smooth curves at compact smooth surface $S$. The cohomology groups of the complement $S\setminus\bigcup C_a$  can be calculated via spectral sequence, if this sequence degenerates in the term $E_2$, one has the following description of the second cohomology group.

It is filtrated by 3-step filtration, the minimal term can be represented by regular differential forms on $S$.

The second is enlarged by differential forms with logarithmic singularities at $C_a$, such that the residues are regular differential 1-forms on $C_a$.
This collection of supported on the curves residues is subject of rather natural restrictions and determines the initial differential form up to addition   regular on  the  surface form.

The third shape of differential forms contains log-forms such that the residues have only log-singularities only at poins of intersections of the curves. For residues of 2-forms this collection of 1-forms on the components is subject of the Parshin reciprocity law: for any point of intersection the sum of residues vanishes.

In the spirit of the theory of the mixed Hodge structures  we shall consider a pair of arrangements and correspond to them single-valued  Aomoto-type dilogarithms according to the following construction: consider a top degree form with a singularity at arrangement. For a pair we can integrate such a form correspondent to the first arrangement against the complex conjugate of a form correspondent to the second.

\subsection{Regular Functions and Differential Forms.}

A holomorphic function  on $X_0(n)$ corresponds to a modular form of the weight 0, so is constant. The holomorphic differential 1-forms on $X_0(n)$ correspond to  cusp forms of the weight 2. The space of such cusp forms is zero for $n=1$.  Hence by the Kunneth formula there are not neither holomorphic 1-forms nor holomorpic  2-forms on $X^2$. Consequently, a form of such type  on the complement of  an arrangement is uniquely determined by its residues at components of arrangement.  
 
The inverse  statement is also true: any collection of 1-forms on the components can be realized as residue of some 2-form on the complement. This follows from the explicit construction of such form.

 %%%%%%%%%%%%%%%%%%%%%%%%%%%%%%%%%%%%%%%%%%%%%%%%%%%%%

\subsection{The Modular Cauchy kernel.}
 
 %%%%%%%%%%%%%%%%%%%%%%%%%%%%%%%%%%%%%%%%%%%%%%%%%%%%%
        
In the paper \cite{Sa15} we have constructed the function $\Xi_0(n)(z,w)$ on $X_0(n)^2\setminus\mbox{diagonal}$. We denote by $(z, w)$  the coordinates on 
$\Bbb H \times \Bbb H,$ and let $p=\exp(2\pi i z)$, $q=\exp(2\pi i w)$. For the matrix $\gamma = \left(\begin{matrix} a&b\\ c&d\end{matrix}\right)$ let $j(\gamma, z) = (cz + d)$ be the automorphic factor. By definition, we put $$\mu_\gamma(z, w)= czw+dw-az-b=(cz+d)(w-\gamma z) =(1~ -z)\gamma^{\ast}\left(\begin{matrix} w\\ 1\end{matrix}\right).$$ We denote by $\mathrm{Ind}_0(n) = \left|\Gamma / \Gamma_0(n)\right|=n \prod_{p|n} \left( 1 + \frac{1}{p} \right)$ the index of $\Gamma_0(n)$  in a group $\Gamma$ and $\mathcal{A}(n)$ be the set of all representatives of right cosets $\Gamma_0(n) \backslash \Gamma$:   \begin{equation*} \Gamma  = \bigcup_{i=0}^{\mathrm{Ind}_0(n)} \Gamma_0(n) \alpha_i.  \label{repres_Gamma}\end{equation*}  

Following the lead of E. Hecke, we consider the series

\begin{definition} Let $s$ is a complex number, then
	\begin{equation}
	   \Xi_0(n)(z, w,s) = \frac{1}{2} \sum_{\gamma \in \Gamma_0(n)}      
            \frac{\overline{\mu_\gamma (z, w)}\overline{\mu_\gamma(z,
		\bar{w})}(\bar{w} - w)}{|\mu(z, w)|^{2s}|\mu_\gamma(z,                  \bar{w})|^{2s}}. \label{Xis}
	\end{equation}
\end{definition}

It was proven in \cite{Sa15} that this series can be analytically continued to the point $s = 1$, therefore we can put $\Xi_0(n)(z, w)=\lim_{s \rightarrow 1}\Xi_0(n)(z,w, s)$. For summation over the modular group let  $\Xi_0(n)(z, w) = \Xi(z, w)$. The function $\Xi_0(n)(z, w)$, considered as a function of the variable $z$, has a simple pole on the curve $T_n=\left\{(z, w) \in \mathbb{H} \times \mathbb{H}|~ w=\gamma z, ~\gamma \in \Gamma_0(n) \right\}$ and for all $w \in \mathbb{H}$ is equal to zero at all cusps of the group $\Gamma_0(n)$. The asymptotic of this nearly holomorphic modular form with respect to $z$ as $\Im z \to  \infty$ is $ 12/(\pi\mathrm{Ind}_0(n) (z-\bar{z}))+ O(p).$

Let $$E_{2, n}^{\infty}(z, s)=\frac{1}{2} \sum_{\gamma \in \Gamma_{\infty}\backslash \Gamma_0(n)}\frac{(c\bar{z}+d)^2}{|cz+d|^{2s+2}} $$ be the Eisenstein series with complex parameter $s$.
The Eisenstein series of weight~2 can be defined as the limit: $E_{2, n}^{\infty}(z)=\lim_{s\rightarrow 1} E_{2, n}^{\infty}(z, s)$. Since $ E_{2, 1}^{\infty}(z) = 1 - 3/(\pi y) + O(p)$ then
\begin{equation}
    \Xi_0^{\ast}(n)(z, w) dz = \left( \Xi_0(n)(z, w) - 2 \pi i  E_{2, n}^{\infty}(z)\right) dz \label{diff_form_Xi}
\end{equation} has a simple pole at the points $z = \gamma w$ and at the point $z=i\infty$ (with asymptotic $-dp/p +O(1)dp$).\\

We denote by ${\mathcal C}_0(n)(z, w)$ the ``Modular Cauchy kernel'':
\begin{equation} {\mathcal C}_0(n)(z, w)=~\Xi_0^{\ast}(n)(z,w) dz+~\Xi_0^{\ast}(n)(w,z) dw \in \Omega^{1,0}_{X_0(n)^2}\left(\mathrm{log}\mbox{(diagonal)}\right). \end{equation}

In the work  K. ~Bringmann and B. ~Kane \cite{BrK}, it was shown that the function $\Xi_0(n)(z, w)$ as a function of $w$ is a polar harmonic Maas form of weight 0, and as a function of $z$ is a polar harmonic form of Maas of weight 2. In paticular  K.~ Bringmann, B.~ Kane, S.~ Lobrich, K.~ Ono, L.~ Rolen in \cite{BKLOR17} proved that  $\Xi_0(n)(z, w)$ can be used to write an explicit formula for the action  of the Ramanujan's theta operator $\Theta(f)$.\\

If we assume that the genus of the group $\Gamma_0(n)$ is zero, and  $J_{\Gamma_0(n)}(z)=p^{-\frac{1}{n}}+O(p^{\frac{1}{n}})$ is the normalized generator of the ring of the modular functions with respect to $\Gamma_0(n)$, then \cite{Sa19}, \begin{equation} {\mathcal C}_0(n)(z, w) = \partial\log~ \left|J_{\Gamma_0(n)}(z)-J_{\Gamma_0(n)}(w) \right|^2 .  \label{cal C_0(n)}\end{equation}
%Suppouse that $\mathcal{T}_{(S)} = \left\{ T_m, T_n, T_l\right\}$.  

 %%%%%%%%%%%%%%%%%%%%%%%%%%%%%%%%%%%%%%%%%%%%%%%%%%%%%

\subsection{Bimodular differential forms whose residues are the cusps forms of weight 2.}

 %%%%%%%%%%%%%%%%%%%%%%%%%%%%%%%%%%%%%%%%%%%%%%%%%%%%%

Now we  will use the construction and properties of the function $\Xi_0(n)(z, w)$ to determine the logarithmic forms on the complement to the Hecke arrangement with given residues by the following consideration.
Let's consider two projections
\[ \begin{tikzcd}
& X_0(n) \arrow[rightarrow]{dr}{\rho}  \\
X \arrow[leftarrow]{ur}{\lambda} 
 && X
\end{tikzcd}\]
for $\lambda: z \to z $ and $\rho: z \to W_n(z) = -1/nz$ given by Fricke involution $W_n$.  We will explore  the map:
\[ \begin{tikzcd}
X_0(n) \times X_0(n), \arrow{d}{\lambda \times\rho} & \text{where} & \text{diagonal}~ \Delta \arrow{d} \\
X \times X  & & \text{Hecke curve}~ T_n
\end{tikzcd}\]

The direct image under this map transforms singularities at the diagonal to singularities at the Hecke curve, compare with the speculations concerning 
$\Xi_n(z, w) $ above.

First we determine the logarithmic forms $(2,0)$-forms with regular residue to the Hecke curve  on $X^{2} \backslash T_n$. There is a one-to-one correspondence between the $(1,0)$-forms on $X_0(n)$ and the elements of the group $H^0(X^{2},\Omega^{2}_{X^{2}}\langle \mathrm{log} \widetilde{T_n}\rangle)$, since as generators $H^0(X^{2},\Omega^{2}_{X^{2}}\left\langle \mathrm{log} \widetilde{T_n}\right\rangle)$ one can take forms whose residues are the cusp forms of weight 2 with respect to the group $\Gamma_0(n)$.
    
Let $f(z)$ be a cusp form of weight 2 and level $n$. Consider the following holomorphic 2-form on the $X_0(n) \times X_0(n)$:
\begin{equation} \widetilde{\mathrm{CoRes}^1}(f) = {\mathcal C}_0(n)(z, w) \wedge (f(z) dz)\end{equation}

We can determine the holomorphic $(1,0)$-form on $X \times X$ with the poles along the Hecke curve $T_n$ as an image under the map $(\lambda\times \rho)_{\ast}$: let \begin{equation}
    \mathrm{CoRes}^1(f) =(\lambda\times \rho)_{\ast}\widetilde{\mathrm{CoRes}^1}(f). 
\end{equation}
Explicitly     
$$\mathrm{CoRes}^1(f) =\sum_{\alpha_i \in \mathcal{A}(n)} \sum_{\beta_j \in \mathcal{A}(n)}  \Xi_0^{\ast}(n) ( W_n \beta_j w, \alpha_i z) f(\alpha_i z) d W_n \beta_j w \wedge d\alpha_i z .$$
The correctness of the definition is verified
by Lemma \ref{lm_2} in the Appendix. It easy to check that this form is closed and regular at the cusp (since the residue is a cusp form of weight 2 with respect to the full modular group and, accordingly, equals 0).

\subsection{Differential 2-forms on the complement of the ``triangle'' on the $X\times X$.}

Denote by $\mathfrak{D}$ the ``triangle'' on the product of two modular curves $X \times X$, formed as a result of the intersection of three Hecke curves $T_n$, $T_m$, $T_l$. 

\begin{center}
    \begin{tikzpicture}[scale = 1,line width=1pt,
        %thick,
        >=stealth',
        dot/.style = {
          draw,
          fill = white,
          circle
        }
      ]
      \coordinate (O) at (0,0);
      
        \draw[->] (-0.3,0) -- (5.5,0) coordinate[label = {below:}] (xmax);
        \draw[->] (0,-0.3) -- (0,5) coordinate[label = {right:}] (ymax);
          
        \draw[black] plot[smooth] coordinates {(0.6,1.3) (1,2) (2,3.3) (3,3.9) (4,4.2) } node [right] {$T_m$} ;
        \draw[black] plot[smooth] coordinates {(0.6, 2.5) (1,2) (2,1.2) (3,0.8) (4,0.75) (4.7,1)} node [right] {$T_n$} ;
        \draw[black] plot[smooth] coordinates {(2.8,4.5)(3,3.9) (3.7,2.7) (4,0.75) (4.5,0.4)} node [right] {$T_l$} ;    
        
        \filldraw[black] (4,0.75) circle (2pt) node[anchor=north east]{$p_{ln}$};
        \filldraw[black] (1,2) circle (2pt) node[anchor=east]{$p_{mn}$};
        \filldraw[black] (3,3.9) circle (2pt) node[anchor= south east]{$p_{ml}$};    
    
        \draw[gray, densely dotted] (1,2) -- (4.7,2);
        \draw[gray, densely dotted] (0.8,0.75) -- (4.7,0.75);
        \draw[gray, densely dotted] (0.8,3.9) -- (4.7,3.9);
        
        \draw[gray, dashdotted] (1,0.4) -- (1,4.5);
        \draw[gray, dashdotted] (4,0.4) -- (4,4.5);
        \draw[gray, dashdotted] (3,0.4) -- (3,4.5);    
        
     %   \node[] at (8,2.5) {arrangement $\mathcal{T}_{(S)}$ on $X \times X$};
        \node[] at (2.3,2.75) {\large$\mathfrak{D}_m$};         
    \end{tikzpicture}   
    \quad\quad\quad
\begin{tikzpicture}[scale = 1,line width=1pt,
        %thick,
        >=stealth',
        dot/.style = {
          draw,
          fill = white,
          circle
        }
      ]
  \coordinate (O) at (0,0);
  
  \draw[->] (-0.3,0) -- (5.5,0) coordinate[label = {below:}] (xmax);
  \draw[->] (0,-0.3) -- (0,5) coordinate[label = {right:}] (ymax);  

    \draw[black] plot[smooth] coordinates {(-0.3,-0.3) (4.5,4.5)} node [right] {$\Delta$} ;
    
    \filldraw[draw=lightgray, fill=lightgray!20, densely dotted] (1.5,1.5) -- (3.5,3.5) -- (3.5,1.5) -- cycle;
    
    \filldraw[black] (1.5,1.5) circle (2pt) node[anchor=north west]{$\widetilde{p_{mn}}$};
    \filldraw[black] (3.5,3.5) circle (2pt) node[anchor=north west]{$\widetilde{p_{nl}}$};

    \draw[black, densely dotted] (0,1.5) -- (5,1.5) node[pos=0, left] {$a$};
    \draw[black, densely dotted] (0,3.5) -- (5,3.5)node[pos=0, left] {$b$};

  \draw[black, dashdotted] (1.5,0) -- (1.5,5);%node[pos=0, below] {$a$};
 \draw[black, dashdotted] (3.5,0) -- (3.5,5); %node[pos=0, below] {$b$};
     
    \node[] at (3,2.25) {\large$\widetilde{\mathfrak{D}}_n$};
  %  \node[] at (6,2) {$X_0(n) \times X_0(n)$};  
  
\end{tikzpicture}
 
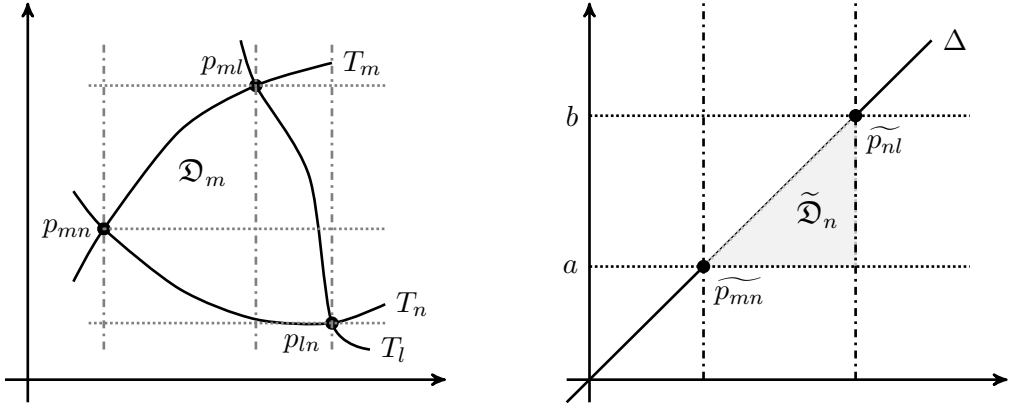
\captionof{figure}{Left: Hecke arrangement on $X \times X$. Right: $X_0(n) \times X_0(n)$. The diagonal $\Delta$ on $X_0(n) \times X_0(n)$ maps to the Hecke curve $T_n$ under the mapping $\lambda \times \rho$.}
\end{center}

On each of the curves $T_m, T_m, T_l$ there exists a form with poles at two intersection points ($p_{mn}, p_{ml}$ or $p_{nl}$).

\begin{definition}  \label{Cores_def_2} Let $\widetilde{p_{nl}}=(a,a),$ $ 
 \widetilde{p_{mn}}=(b,b)$ be a preimages of CM points $p_{nl}, p_{mn}$ on $\Delta \subset X_0(n) \times X_0(n)$. We define the 2-form with poles on the diagonal and at the points $\widetilde{p_{nl}},$ $ \widetilde{p_{mn}}$:
\begin{equation*}
\widetilde{\mathrm{CoRes}^2}(n, a, b) = {\mathcal C}_0(n)(z, w) \wedge \left( {\mathcal C}_0(n)(z, a) - {\mathcal C}_0(n)(z, b) \right). \label{Cores_def_1}
\end{equation*}
Let \begin{equation}
    \mathrm{CoRes}^2(n, a, b) =(\lambda\times \rho)_{\ast}\widetilde{\mathrm{CoRes}^2}(n, a, b).
\end{equation}
\end{definition}

By the definition, the action of the Fricke involution on the modular forms of weight 2 for the subgroup $\Gamma_0(n)$ is described by the formula: $f(z) \circ \left[W_n\right] = 1/(nz^2)f(W_n z)$, so $W_n f(z) = n f(z) \circ \left[W_n\right]$.% \footnote{S. Lang, Introduction to Modular Forms, p.121}. 
Therefore, rewriting the desired form $\mathrm{CoRes}^2(n, a, b)$ on the product of two modular curves in terms of the $\Xi$-functions we have
\begin{multline}
\mathrm{CoRes}^2(n, a, b) =\\= n \sum_{\alpha_i \in \mathcal{A}(n)}\sum_{\beta_j \in \mathcal{A}(n)}  \Xi^{\ast}_0(n) ( W_n \beta_j w,  \alpha_i z) d  W_n \beta_j w \wedge \Xi^{\ast}_0(n) (\alpha_i z,   W_n a) d \alpha_i z  -\\
     -  \Xi^{\ast}_0(n) (W_n \beta_j w,  \alpha_i  z) d W_n \beta_j w \wedge \Xi^{\ast}_0(n) (\alpha_i z,  W_n b) d \alpha_iz. 
\end{multline}
There are several supporting statements about the poles and residues of this form:
\begin{lemma} The form $\mathrm{CoRes}^2(n, a, b)$ has the following singularities:
the simple pole at $w=0$ $($with the asymptotic $\left(-dq/q +O(1) dq \right) \wedge \left( \Xi^{\ast}(z, a)d z-\Xi^{\ast}(z, b)dz \right))$,
the simple pole at the image of the points $z=a, b$ with the residue $\Xi^{\ast}(w, a)dw$ and at the Hecke curve $T_n$ with the residue $\Xi^{\ast}(w, a)dw - \Xi^{\ast}(w, b)dw.$
\end{lemma}
\begin{proof}
\begin{enumerate}
\item It was proven in \cite{Sa19} that for $\Im z \to \infty$, $\Xi_0(n)(w, z)$ can be represented by the power series in terms of the Poincare series of weight 2: 
\vspace{-0.3cm}
    \begin{equation*}
	\Xi_0(n)( w, z)=2 \pi i  \lim_{s\rightarrow 1} \left( E_{2, n}^{\infty}(w, s) + \sum_{r> 0}    P^{\infty}_{n, r}( w, s)~p^{r} + \sum_{r< 0}P^{\infty}_{n, r}( w,  s)~\bar{p}^{r} \right), 
    \end{equation*}
\vspace{-0.5cm}
where $$P^{\infty}_{n, k}(z, s)=\frac{1}{2}\sum_{\gamma \in \Gamma_{\infty} \backslash \Gamma_0(N)}\frac{e^{-2 \pi i k \gamma z_1}}{(cz+d)^2|cz+d|^{2s-2}}.$$
So, the constant term  of the function $\Xi^{\ast}_0(n)( W_n \beta w,  z) d  W_n \beta w$ coincides with the constant term of the form $2 \pi i  \lim_{s\rightarrow 1} E_n^{\infty}( W_n \beta w, s) d  W_n \beta w$. Further, since $W_n \left(\begin{matrix} a & b \\ c& d \end{matrix}\right)   = \left(\begin{matrix} -c & -d \\ an& bn \end{matrix}\right)$, then 
$\sum_{\beta \in \mathcal{A}(n)} E_{2, n}^{\infty}( W_n \beta w) d  W_n \beta w = n^{-1} E_{2, 1}^{\infty}(w) d w$ is the Eisenstein series of weight 2 with respect to the full modular group. This completes the proof of the first statement.

\item By the definition, the residue at the points $\alpha_i z = W_n a$ is the form: 
$$ \sum_{\alpha_i \in \mathcal{A}(n)} \sum_{\beta_j \in \mathcal{A}(n)}  \Xi^{\ast}_0(n) ( W_n \beta_j w, h^{\ast} W_n \beta_j  a) d W_n \beta_j w,  ~~~ h \in \Gamma_0(n). $$
Further, for $h, g \in \Gamma_0(n)$ we have
$$ \mu_{g}( W_n \beta_j w, h^{\ast} W_n a) = \mu_{W_n^{\ast}  h g W_n \beta_j}(w,  a)(j(W_n \beta_j,w))^{-1}(j(h W_n ,a))^{-1}. $$
Note that
$W_n^{\ast}  h g W_n \beta_j =  W_n^{\ast}  W_n  g' \beta_j =  W_n^{\ast}  W_n\gamma = n \gamma, $ where
 $\gamma$ is a some representative of the coset (for $\beta_j$). Therefore,
$\mu_{ W_n^{\ast}  h g W_n \beta_j}(w, a)  =   n \mu_{\gamma}(w, a). $ So, after summation over $\beta_j \in \mathcal{A}(n)$, we obtain $$\mathrm{Res}_{\alpha_i z = W_n a}\mathrm{CoRes}^2(n, a, b) =  \Xi^{\ast}(w, a)dw.$$

\item Finally, we find the residue of $\mathrm{CoRes}^2(n, a, b)$ at the Hecke curve $T_n$: 
\begin{multline*}
    \mathrm{Res}_{T_n}\mathrm{CoRes}^2(n, a, b) = n \sum_{\beta_j \in \mathcal{A}(n)}  \left( \Xi^{\ast}_0(n) (W_n \beta_j w,   W_n a) d W_n \beta_j w  \right. -\\- \left.  \Xi^{\ast}_0(n) (W_n \beta_j w,  W_n b) d W_n \beta_j w \right) = \Xi^{\ast}(w, a)dw - \Xi^{\ast}(w, b)dw.
\end{multline*}
\end{enumerate}
\end{proof}
\begin{definition} \label{Def_Cores_2}  The closed differential 2-form on the complement to the ``triangle'' $\mathfrak{D}$ formed by the arrangement of the Hecke curves, i.e. ``sides'' $T_m, T_n, T_l$, is
\begin{equation} 
 \mathrm{CoRes}^2(\mathfrak{D})  =\mathrm{CoRes}^2(n, a, b) - \mathrm{CoRes}^2(m, a, c) +\mathrm{CoRes}^2(l, b, c), 
 \end{equation} where $a, b, c$ is the coordinates of the corresponding intersection points (i.e. the vertices of the ``triangle'').
\end{definition}

It obviously follows from the previous lemma that this form is holomorphic in the cusps and in the images of the points $z=a, b$. At the same time, this form has simple poles on the Hecke curves with residues $\mathrm{Res}_{T_n}\mathrm{CoRes}^2 = \Xi^{\ast}(w, a)dw - \Xi^{\ast}(w, b)dw$. The $\bar{\partial}$-closedness of this form follows naturally from the following assertion \cite{Sa19}: 
\begin{equation*}
		\lim_{s \rightarrow 1} \frac{d}{d\bar{z}}~\Xi(w,z,s)=-\frac{1}{2}~\sum_{\gamma \in \Gamma}\frac{1}{\mu_{\gamma}(w,- \bar{z})^2}=-\omega(w, \bar{z}), \label{eq: Dz_2Xi}
	\end{equation*}
where the function $\omega(w,\bar{z})$ was introduced by Don Zagier in \cite{Za75}. It is easy to show that the function $\omega(w,\bar{z})$ is a cusp form with respect to $w$ of weight 2, and therefore is equal to 0.

This construction can be generalized to any cyclic chain of the Hecke curves (``polygon'').

\section{The Hauptmodule Modular Aomoto Dilogarithm} 
\subsection{The Regulator Formulas.}
Let $X$ be a smooth complex variety. Following to Goncharov \cite{Go02}, correspond to a collection of holomorphic non-vanishing functions $\{f_i\}$  the differential forms:
\begin{equation}r_2(f_1,f_2)=\frac12\left(\log|f_1|^2( \partial-\overline{\partial})\log|f_2|^2
-\log|f_2|^2( \partial-\overline{\partial})\log|f_1|^2\right),\end{equation}
%\vspace{-0.8cm}
\begin{multline} \quad\quad r_3(f_1,f_2,f_3)=\\=\frac16\sum_{\sigma\in {\frak S}_3}{\rm sign}(\sigma) \phi_{\sigma(1)} \left(\partial\phi_{\sigma(2)}\wedge  \partial\phi_{\sigma(3)}-\partial\phi_{\sigma(2)}\wedge \overline{\partial}\phi_{\sigma(3)}+
\overline{\partial}\phi_{\sigma(2)}\wedge\overline{\partial}\phi_{\sigma(3)}\right),\end{multline}
where $\phi_j=\log|f_j|^2.$ Then
$$d D_2(f)=r_2(f,1-f),$$ 
%\vspace{-0.5cm}
$$d r_2(f_1,f_2)={\partial}\log|f_1|^2\wedge{\partial}\log|f_2|^2-
\overline{\partial}\log|f_1|^2\wedge\overline{\partial}\log|f_2|^2,$$
%\vspace{-0.5cm}
%$$-2\pi i\sum_D {\rm ord}_D(f_{1})\log|f_2|^2\delta_D 
%+2\pi i\sum_D {\rm ord}_D(f_{2})\log|f_1|^2\delta_D;
$$dr_3(f_1,f_2,f_3)={\partial}\log|f_1|^2\wedge{\partial}\log|f_2|^2\wedge{\partial}\log|f_3|^2+
\overline{\partial}\log|f_1|^2\wedge\overline{\partial}\log|f_2|^2\wedge\overline{\partial}\log|f_3|^2.
$$

\subsection{Compact Curves.}

For a compact curve $C$ and meromorphic functions $f$ and
$g$ the form $r_2(f,g)$ is a current and
$$d r_2(f,g)=
-2\pi i\sum_{p\in C} \log|{\rm ts}_p(f,g)|^2\delta_p,
$$
where ${\rm ts}_p(f_1,f_2)$ denotes the tame symbol 
$$\left.\frac{f^{{\rm ord}_p(g)}}{g^{{\rm ord}_p(f)}}\right|_p.$$

 For  meromorphic functions $f_1, f_2, f_3$,  such that the divisor of $f_3$ does not intersect divisors of  $f_1$ and $f_2$, the forms $r_2(f_1,f_2)\wedge \overline{\partial}\log|f_3|^2$ and $r_3(f_1,f_2,f_3)$ are currents and 
 \begin{multline} r_2(f_1,f_2)\wedge \overline{\partial}\log|f_3|^2+r_3(f_1,f_2,f_3)=\\= d\left(\frac13(\phi_1\phi_3\partial \phi_2-\phi_2\phi_3\partial \phi_1)+\frac16 (\phi_2\phi_3\overline{\partial}\phi_1-\phi_1\phi_3\overline{\partial}\phi_2)\right)-\\
-\frac{2\pi i}{2} \sum_{p\in C}\left(\mathrm{ ord}_p(f_2)(\phi_1(p)\phi_3(p))-\mathrm{ ord}_p(f_1)(\phi_2(p)\phi_3(p))\right)\delta_p,
\end{multline}
here, as above 
$\phi_i=\log|f_i|^2
$. 
\begin{multline} \int_{\Bbb P^1}r_3(x,1-x,x-t)
=\\=-\int_{\Bbb P^1}r_2(x,1-x)\wedge \overline{\partial}\log|x-
t|^2=-\int_{\Bbb P^1}dD_2(x)\wedge \overline{\partial}\log|x-t|^2=\\=-\int_{\Bbb P^1}d\left(D_2(x)\ \overline{\partial}\log|x-t|^2\right)+\int_{\Bbb P^1}D_2(x)
d\overline{\partial}\log|x-t|^2=-2\pi i D_2(t). \end{multline}

 \textbf{Fact (Rudenko)} \cite{Ru15} \label{Fact_R}. For any curve $C$ and triple of meromorphic functions $\{f,g,h\}$ the integral $\frac1{2\pi i}\int_{C}r_3(f,g,h)$ is an algebraic sum of the values of $D_2$:
 
\begin{equation}\sum_jm_j 2\pi i D_2(\eta_j),~~~ \sum_jm_j\eta_j={\mathcal H}(f,g,h),~~~\text{such that}\end{equation}\begin{gather*}
\sum_jm_j\eta_j\wedge(1-\eta_j)=\sum_{p\in C}
\left(\mathrm{ord}_p(f)(g(p)\wedge h(p))- \mathrm{ord}_p(g)(f(p)\right.\wedge h(p)) + \\ \left. +~ \mathrm{ord}_p(h)(f(p)\wedge g(p))\right)\in \bigwedge^2(\Bbb C^*).
\end{gather*}

So,
\begin{multline}  \int_Cr_2(f_1,f_2)\wedge \overline{\partial}\log|f_3|^2=-\int_Cr_3(f_1,f_2,f_3) +\\
\int_C d\left(\frac13(\phi_1\phi_3\partial \phi_2-\phi_2\phi_3\partial \phi_1)+\frac16 (\phi_2\phi_3\overline{\partial}\phi_1-\phi_1\phi_3\overline{\partial}\phi_2)\right)-\\
-\frac{2\pi i}{2} \sum_{p\in C}\left(\mathrm{ord}_p(f_2)(\phi_1(p)\phi_3(p)-\mathrm{ord}_p(f_1)(\phi_2(p)\phi_3(p)\right).
\end{multline}

\subsection{The Square of the Modular Curve. Modular Aomoto Dilogarithm.} 

We call a pair of arrangements $\frak{D}$ and $\frak{D}'$ {\em admissible } if their elements $T_i$ are mutually different and intersection points of elements of each arrangement are mutually different also. In this case the integral 
$$\int_{Y^2}{\mathrm{CoRes}^2}(\frak{D}')\wedge \overline{\mathrm{CoRes}^2(\frak{D})}$$ converges and is known as the Aomoto dilogarithm.

We shall restrict ourself by the case when all modular correspondence are curves of genus zero. In this case the filtration on the cohomology group of the complement reduces to the  forms last shape, there aren't neither regular forms nor forms with regular residues. We call such modular arrangements {\em Hautmodule} arrangements.

\begin{theorem} For admissible pair of the Hauptmodule modular arrangements $\frak{D}$ and $\frak{D}'$ the Aomoto dilogarithm
$$\int_{Y^2}{\mathrm{CoRes}^2}(\frak{D}')\wedge \overline{\mathrm{CoRes}^2(\frak{D})}$$
equals to the combination with rational coefficients of the values of the Bloch-Wigner dilogarithm $(2\pi i)^2D_2$ at algebraic number.
\end{theorem}
\begin{proof} By definitions \ref{Cores_def_2}, \ref{Def_Cores_2}
$$ \mathrm{CoRes}^2(\mathfrak{D})=\mathrm{CoRes}^2(n, a, b) - \mathrm{CoRes}^2(m, a, c) +\mathrm{CoRes}^2(l, b, c), ~~{\mathrm{CoRes}^2}(n, a, b)=$$ $$=(\lambda(n)\times\rho(n))_*\left(\partial \log~ \left|J_{\Gamma_0(n)}(z)-J_{\Gamma_0(n)}(w) \right|^2\wedge   \partial\log~ \left|\frac{J_{\Gamma_0(n)}(z)-J_{\Gamma_0(n)}(a) }{J_{\Gamma_0(n)}(z)-J_{\Gamma_0(n)}(b)} \right|^2\right),
   %\left(d\loсозданаg~ \left|J_{\Gamma_0(n)}(z)-J_{\Gamma_0(n)}(a) \right|^2-d \log~ \left|J_{\Gamma_0(n)}(z)-J_{\Gamma_0(n)}(b) \right|^2\right), 
   $$
where the number in parenthesis of the maps $\lambda(n)$ and $\rho(n)$  denotes the level of the source of the maps. \\

For $\mathrm{CoRes}^2(\mathfrak{D})'$ we use another presentation with interchanging variables: 
$$ \mathrm{CoRes}^2(\mathfrak{D}') =\mathrm{ CoRes}^2(n', a', b')^T - \mathrm{CoRes}^2(m', a', c')^T +\mathrm{CoRes}^2(l', b', c')^T,$$ 
$${\mathrm{CoRes}^2}(n', a', b')^T =(\rho(n')\times\lambda(n'))_*\left(\partial \log~ \left|J_{\Gamma_0(n')}(w)-J_{\Gamma_0(n')}(z) \right|^2\right.\wedge $$
$$\left. \wedge ~ \partial\log~ \left|\frac{J_{\Gamma_0(n')}(w)-J_{\Gamma_0(n')}(a') }{J_{\Gamma_0(n')}(w)-J_{\Gamma_0(n')}(b')} \right|^2\right).
   %\left(d\log~ \left|J_{\Gamma_0(n)}(z)-J_{\Gamma_0(n)}(a) \right|^2-d \log~ \left|J_{\Gamma_0(n)}(z)-J_{\Gamma_0(n)}(b) \right|^2\right), 
   $$
Consider the  fiber product ${\mathcal Y}^{(2)}(n',n)$ of the coverings $(\rho(n')\times\lambda(n'))$ and $\lambda(n)\times\rho(n)$:
\begin{comment}
 \begin{equation} \begin{array}{ccccc}
   && {\mathcal Y}^{(2)}(n',n)\\
   &&\\
   &\stackrel{\pi(n')}{\swarrow}&&\stackrel{\pi(n)^T}{\searrow}&\\
   Y_0(n')\times Y_0(n')&&\downarrow \pi&&Y_0(n')\times Y_0(n')\\
&\stackrel{\rho(n')\times\lambda(n')}{\searrow}&&\stackrel{\lambda(n)\times\rho(n)}{\swarrow}&\\
  &&Y\times Y&&
 \end{array} \end{equation} 
\end{comment}
\begin{equation}
 \begin{tikzcd}
& {\mathcal Y}^{(2)}(n',n) \arrow{dr}{\pi(n)^T} \arrow[swap]{dl}{\pi(n')} \arrow{dd}{\pi} & \\
 Y_0(n')\times Y_0(n') \arrow[swap]{dr}{\rho(n')\times\lambda(n')} & & Y_0(n)\times Y_0(n)\arrow{dl}{\lambda(n)\times\rho(n)}\\
&Y\times Y&
\end{tikzcd}
\end{equation}
 
Consider the following functions on $ {\mathcal Y}^{2}(n',n)$:

$$F(n, a, b)=\pi(n)^*\left(J_{\Gamma_0(n)}(z)-J_{\Gamma_0(n)}(w) \right),$$
\begin{equation} f(n, a, b)=\pi(n)^*\left(\frac{J_{\Gamma_0(n)}(z)-J_{\Gamma_0(n)}(a) }{J_{\Gamma_0(n)}(z)-J_{\Gamma_0(n)}(b)}  \right),\end{equation}
$$G(n', a', b')=\left(\pi(n')^T\right)^*\left(J_{\Gamma_0(n')}(w)-J_{\Gamma_0(n')}(z) \right),$$
$$g(n', a', b')=\left(\pi(n')^T\right)^*\left(\frac{J_{\Gamma_0(n')}(w)-J_{\Gamma_0(n')}(a') }{J_{\Gamma_0(n')}(w)-J_{\Gamma_0(n')}(b')}  \right).
$$

Then
\begin{equation*}{\mathrm{CoRes}^2}(\frak{D}')\wedge \overline{\mathrm{CoRes}^2(\frak{D})}=\pi_*
\left(\partial\log|G|^2\wedge\partial\log|g|^2\wedge\overline {\partial}\log|F|^2\wedge\overline{\partial}\log|f|^2\right),
\end{equation*}
\begin{equation} d \left(r_2(G,g)\wedge\overline {\partial}\log|F|^2\wedge\overline{\partial}\log|f|^2\right)=
\end{equation}
\begin{gather*} =
\left(\partial\log|G|^2\wedge\partial\log|g|^2\wedge\overline {\partial}\log|F|^2\wedge\overline{\partial}\log|f|^2\right)
+\\+ 2\pi i\sum_C\mathrm{ord}_C (F)\left(r_2(G,g)\wedge\overline{\partial}\log|f|^2\right)|_C\delta_C - \\ - 2\pi i\sum_C\mathrm{ord}_C(f)\left(r_2(G,g)\wedge\overline{\partial}\log|G|^2\right)|_C\delta_C.
\end{gather*} 

So, by the Stocks formula, we reduce the integral to one-dimensional. Passing to finite covering by enlarging level we get smooth domain of integration. Now we apply result of Rudenko  mentioned in the subsection \ref{Fact_R}. 
\end{proof} 

 \vspace{0.5cm}

\textbf{Acknowledgments.} The article was prepared within the framework of the project “International academic cooperation” HSE University.

 \vspace{0.5cm}

 \textsc{National Research University Higher School of Economics,}
 
 \textsc {Laboratory of Mirror Symmetry, NRU HSE,}
 
 \textsc {6 Usacheva str., Moscow, Russia, 119048}\par\nopagebreak
  \textit{Email addresses:} A. Levin: \texttt{alevin@hse.ru},  N.~Sakharova: \texttt{nsakharova@hse.ru}

\newpage
\appendix \section{Auxiliary lemma Concerning the Modular Cauchy kernel.}

Denote by 
	\begin{equation}
	   \Xi_n(z, w) = \frac{1}{2} \sum_{\gamma \in {\rm Mat}'_n}      
            \frac{\overline{\mu_\gamma (z, w)}\overline{\mu_\gamma(z,
		\bar{w})}(\bar{w} - w)}{|\mu(z, w)|^{2s}|\mu_\gamma(z,                  \bar{w})|^{2s}}.  \label{Xin}
	\end{equation}
The differential form ${\mathcal C}_n(z, w)=~\Xi_n(z,w) dz+~\Xi_n(w, z) dw \in \Omega^{1,0}_{Y^2}\left(\mathrm{log}T_n\right). $

\begin{lemma} \label{lm_2}  Let $W_n= \left(\begin{matrix} 0& 1 \\ -n & 0 \end{matrix} \right)$ be the matrix of the Fricke involution and $\alpha_i, \beta_j \in \Gamma_0(n) \backslash \Gamma$ are the representatives of right cosets, then
\begin{equation} 
\Xi_1(z, w) = \sum_{\alpha_i \in \mathcal{A}(n)} \Xi_0(n) (\alpha_i z, w) j^{-2}(\alpha_i, z)  ~~\text{and} \label{Xi_descent1} 
\end{equation}
\begin{equation} 
\Xi_n(z, w) = \sum_{\alpha_i \in \mathcal{A}(n)} \sum_{\beta_j \in \mathcal{A}(n)} \Xi_0(n) (\alpha_i z, W_n \beta_j w) j^{-2}(\alpha_i, z) j^{-2}(W_n \beta_j, w).\label{Xi_descent2} 
\end{equation}
\end{lemma}
The proof of the first equality is quite standard. The second statement is a direct consequence of Smith's normal form.


\begin{thebibliography}{XXX}
%\addcontentsline{toc}{section}{Bibliography}

 \bibitem[Be86]{Be86}
A. A. Beilinson, {\it Higher regulators of modular curves. Applications of algebraic K-theory to algebraic geometry and number theory\/}. Contemporary Mathematics 55, 1-34 (1986).

\bibitem[BrK16]{BrK}
 K. Bringmann and B. Kane, {\it A problem of Petersson about weight 0 meromorphic modular forms\/}. Research in
 Mathematical Sciences (2016).

\bibitem[BKLOR17]{BKLOR17}
 K. Bringmann, B.~ Kane, S.~ Lobrich, K.~ Ono, L.~ Rolen, {\it On divisors of modular forms\/}, \href{https://arxiv.org/pdf/1609.08100}\href{https://arxiv.org/pdf/1609.08100} (2017).

\bibitem[BD19]{BD19}
F. Brown, C. Dupont, {\it Single-valued integration and superstring amplitudes in genus zero\/}, \href{https://arxiv.org/abs/1910.01107}\href{https://arxiv.org/abs/1910.01107} (2019).

%\bibitem[Bo95]{Bo95}
% R. E. Borcherds, {\it Automorphic forms on $O_{s + 2,2}(\mathbb{R})$ and infinte products\/}. Invent. Math. \textbf{120}, 161-214 (1995).


\bibitem[Go02]{Go02} A. B. Goncharov, {\it Polylogarithms, regulators, and Arakelov motivic complexes\/}, \href{arXiv:math/0207036v3}{arXiv:math/0207036v3} (2002).


  \bibitem[KZ01]{KZ01}
M. Kontsevich, D. Zagier, {\it Periods, Mathematics unlimited -- 2001 and beyond\/}. Springer, Berlin, 771-808 (2001).


\bibitem[Ru15]{Ru15} D. Rudenko, {\it The Strong Suslin Reciprocity Law\/},
\href{arXiv:1511.00520v3}{arXiv:1511.00520v3} (2015).


\bibitem[Sa15]{Sa15}
 N. Sakharova, {\it Convergence of the Zagier type series for the Cauchy kernel\/}, \href{https://arxiv.org/abs/1503.05503}{https://arxiv.org/abs/1503.05503} (2015).

\bibitem[Sa19]{Sa19}
 N. Sakharova, {\it Modular Cauchy kernel corresponding to the Hecke curve\/}. Arnold Mathematical Journal, 4(3), 301-313 (2019).


\bibitem[Za75]{Za75}
 D. Zagier, {\it Traces des opérateurs de Hecke\/}. Séminaire Delange-Pisot-Poitou. Théorie des nombres, Volume: 17, Issue: 2, page 1-12 (1975-1976).


%\bibitem[ZaGr85]{ZaGr85}
% D. Zagier, B.  Gross, {\it On singular moduli\/}. J. reine Angew. Math., \textbf{355}, 191-220 (1985).

\bibitem[ZaGr86]{ZaGr86}
 D. Zagier, B.  Gross, {\it Heegner points and derivative of L-series\/}. Invent. Math., \textbf{85}, 225-320 (1986).
 \end{thebibliography}
\end{document}